\documentclass[12pt]{amsart}

\usepackage[ansinew]{inputenc}
\usepackage[english]{babel}
\usepackage{anysize}
\usepackage{amsfonts}
\usepackage{amssymb}
\usepackage{amsmath}
\usepackage{amsthm}
\usepackage{doc}
\usepackage{exscale}
\usepackage{fontenc}
\usepackage{ifthen}
\usepackage{latexsym}
\usepackage{graphicx}
\usepackage{float}
\usepackage{array}
\usepackage[all]{xy}
\usepackage[usenames,dvipsnames]{color}
\usepackage{hyperref}
\usepackage{enumerate}
\usepackage{multirow}
\usepackage{color}
\usepackage{microtype}
\usepackage{url}
\usepackage{varioref}
\usepackage{dsfont}
\usepackage[colorinlistoftodos,prependcaption,textsize=tiny]{todonotes}
\usepackage{xargs} 

\newcommandx{\unsure}[2][1=]{\todo[linecolor=red,backgroundcolor=red!25,bordercolor=red,#1]{#2}}

\newtheorem{thm}{Theorem}[section]

\newtheorem{Lem}[thm]{Lemma}

\newtheorem{prop}[thm]{Proposition}

\newtheorem{rmk}[thm]{Remark}

\newtheorem{lem-df}[thm]{Lemma-Definition}

\newtheorem{rmk-df}[thm]{Remark-Definition}

\newtheorem*{thm*}{Theorem}

\newcommand{\beq}{\begin{equation}}
\newcommand{\enq}{\end{equation}}
\newcommand{\beqn}{\begin{equation*}}
\newcommand{\enqn}{\end{equation*}}

\newcommand{\mD}{\mathbb{D}}

\newcommand{\caE}{\mathcal{E}}
\newcommand{\caF}{\mathcal{F}}

\newcommand{\caH}{\mathcal{H}}

\newcommand{\Z}{\mathbb{Z}}

\newcommand{\C}{\mathbb{C}}

\newcommand{\PP}{\mathbb{P}}
\newcommand{\OO}{\mathcal{O}}

\newcommand{\M}{\mathcal{M}}
\newcommand{\A}{\mathcal{A}}
\newcommand{\cH}{\mathcal{H}}
\newcommand{\m}[1]{\mathcal{#1}}
\newcommand{\pder}[1]{\frac{\partial}{\partial #1}}

\DeclareMathOperator{\Alb}{Alb}
\DeclareMathOperator{\alb}{alb}

\DeclareMathOperator{\Sym}{Sym}

\DeclareMathOperator{\h}{h}

\DeclareMathOperator{\Ker}{Ker}
\DeclareMathOperator{\Id}{Id}
\DeclareMathOperator{\Coker}{Coker}

\DeclareMathOperator{\Nm}{Nm}

\DeclareMathOperator{\Res}{Res}

\DeclareMathOperator{\Deg}{Deg}

\DeclareMathOperator{\dPE}{\mbox{\textit{d}}_o \textit{P}_E}
\DeclareMathOperator{\dP}{\mbox{\textit{d}}_o \textit{P}}
\DeclareMathOperator{\dPEv}{\mbox{\textit{d}}_o \textit{P}_E^{\vee}}
\DeclareMathOperator{\dPv}{\mbox{\textit{d}}_o \textit{P}^{\vee}}

\DeclareMathOperator{\spr}{\hat{\otimes}}

\begin{document}

\title{Covering of elliptic curves and\\ 
the kernel of the Prym map}
\date{\today}

\author{Filippo F. Favale}
\address[Filippo F. Favale]{Department of Mathematics, University of Trento, via Sommarive 14,
I-38123 Trento, Italy}
\email{filippo.favale@unitn.it}
\author{Sara Torelli}
\address[Sara Torelli]{Department of Mathematics, University of Pavia, Via Ferrata, 5,
I-27100 Pavia, Italy}
\email{sara.torelli02@ateneopv.it}



\begin{abstract}
Motivated by a conjecture of Xiao, we study families of coverings of elliptic curves and their corresponding Prym map $\Phi$. More precisely, we describe the codifferential of the period map $P$ associated to $\Phi$ in terms of the residue of meromorphic $1$-forms and then we use it to give a characterization for the coverings for which the dimension of $\Ker(dP)$ is the least possibile. This is useful in order to exclude the existence of non isotrivial fibrations with maximal relative irregularity and thus also in order to give counterexamples to the Xiao's conjecture mentioned above. The first counterexample to the original conjecture, due to Pirola, is then analysed in our framework.
\vspace{4mm}

\end{abstract}

\maketitle  


\section*{Introduction}
\label{Intro}

\noindent Hurwitz spaces were classically introduced and studied by Clebsh and Hurwitz (see \cite{Cle} and \cite{Hur}) as spaces parametrizing branched coverings of $\PP^1$. Nowadays, the term Hurwitz space refers to a variety which parametrizes, up to equivalence, coverings $\pi:F\rightarrow E$ of curves with some geometric restrictions. In this article we will use a local version of Hurwitz spaces, namely a {\it local family of coverings}, whose seminal idea can be found in \cite{Kan}. Roughly, given a fixed covering $\pi:F\rightarrow E$ where $E$ is an elliptic curve, one is able to construct a map $p:\caF\rightarrow \caE$ of curves over $\cH$, where $\cH$ is a contractible open set. Then $\cH$ is a parameter space for smooth coverings which share the same degree and the same ramification indices with $\pi$.
\vspace{2mm}

\noindent Attached to a local family of coverings $p:\caF\rightarrow \caE$ with parameter space $\cH$ there is the Prym map $\Phi$, which associates to a $b\in\cH$ the generalized Prym variety of $\pi_b=\pi|_{\caF_b}:\caF_b\rightarrow \caE_b$, i.e.  the connected component containing $0$ of the kernel of the norm map $\Nm(\pi_b)$. The Prym map is, in some sense, the analogous of the Torelli map $T$ from $\M_g$, the moduli space of curves of genus $g$, to $\A_g$, the moduli space of principally polarized abelian varieties of dimension $g$.
\vspace{2mm}

\noindent A celebrated theorem, the infinitesimal Torelli theorem, states that the differential of the Torelli map is injective outside the hyperelliptic locus of $\M_g$ and it should be interesting to have a similar theorem also for Prym maps or, at least, to their lifting $P$ to a period domain. We will show that, in our case, i.e. when the base $E$ is an elliptic curve, the dimension of the kernel of $\dPv$ is at least $1$ as a consequence of how the local families that we will use are constructed. Roughly, by composing a covering with a traslation of the base we always have coverings with the same Prym, so there is a tangent direction in the parameter space along which the Prym map is constant. Hence a question analogous to the one answered by the infinitesimal Torelli is
\vspace{2mm}

\begin{center}
{\em Given a family of coverings with central fiber $\pi$, which conditions\\ 
can we put on $\pi$ in order to have that $\Ker(\dP)$ has dimension $1$?}
\end{center}

\noindent It is already known that an infinitesimal Torelli-like theorem for the Prym map cannot hold without restrictions as there are examples of coverings $\pi:F\rightarrow E$ (moreover with $F$ non hyperelliptic) for which there are two independent directions along which $\dP$ is $0$. One of these examples, due to Pirola, will be analyzed in Section \ref{SEC:4}. This paper is devoted to the study of the Prym map $\Phi:\cH\rightarrow \A$ in the cases for which $\cH$ parametrizes coverings over an elliptic curve. 
\vspace{2mm}

\noindent A further motivation to study this kind of problems comes from a conjecture about fibered surfaces. Recall that, given a fibration $f:S\rightarrow B$ of a smooth compact surface $S$ over a smooth compact curve $B$, the {\em relative irregularity} $q_f$ is defined to be the difference $q(S)-g(B)$. A modified version of a conjecture of Xiao states that, if $f$ is not isotrivial, then
\begin{equation}
q_f\leq \left\lceil\dfrac{g+1}{2}\right\rceil.
\end{equation}
The original conjecture was without the round up and has been modified after a counterexample of Pirola, the one that we will present in Section \ref{SEC:4}. To have an insight of what is known about the relative irregularity and about recent results about an upper bound a good reference is \cite{BGN}. The link between the world of non isotrivial fibrations and the one of the families of coverings we will define is broadly given as follows. The fibration $f$ induces a surjective map $\alb(f):\Alb(S)\rightarrow \Alb(B)=J(B)$ with $\dim(\Ker(\alb(f)))=q_f$, which has a connected component containing $0$. We shall denote it with $K_f$. If $B^0$ is the open subset of $B$ over which the fibration has smooth fibers, we denote by $F_b$ the fiber over $b\in B^0$. Via the map $F_b\hookrightarrow S$ we have a map $JF_b\rightarrow \Alb(S)$ whose image is, up to translation, exactly $K_f$. Dualizing we have a map
$$
\xymatrix{
K_f^{\vee}\ar@{^{(}->}[r] & JF_b^{\vee}=JF_b
}$$
Note that $K_f^{\vee}$ doesn't depend on $b$ whereas $F_b$ strongly depends on it. In particular we have proved that the Jacobian of every smooth fiber of a non isotrivial fibration contains a fixed abelian variety of dimension $q_f$. Assume now that we are in an extreme case, i.e., assume that $q_f=g-1$. Since in this article we are only interested in non isotrivial fibrations, we will call fibration with {\em maximal relative irregularity} those with $q_f=g-1$. In fact, every fibration satisfies $0\leq q_f\leq g$ and the equality $q_f=g$ holds if and only if the fibration is trivial (this follows from a result of Beauville: see the appendix of \cite{Deb} for details).
In this case $\dim(K_f^{\vee})=q_f=g-1$ and we can consider the quotient $JF_b/K_f^{\vee}$ which will be an abelian variety of dimension $g-q_f=1$: an elliptic curve $E_b$. 
$$
\xymatrix{
\Phi(\pi_b)=K_f^{\vee}\ar@{^{(}->}[r] & JF_b\ar@{->>}[r]^-{p_b} & E_b \\
& F_b \ar[ur]_-{\pi_b}\ar[u]\ar[r]_-{\pi_b} & E_b\ar@{=}[u]
}$$
Moreover, in this case $K_f^{\vee}$ is the connected component through the origin of the kernel of the norm map associated to  the ramified covering $\pi_b:F_b\rightarrow E_b$, i.e. the Prym variety $\Phi(\pi_b)$.
Hence, an eventual counterexample to the modified version of the conjecture of Xiao, under the additional assumption $q_f=g-1$, would give a family of coverings of elliptic curves with constant Prym variety. At the moment the question
\begin{center}
{\em Is there a non isotrivial fibration (with maximal relative irregularity $q_f$ or not)\\ giving a counterexample to the modified Xiao's conjecture?}
\end{center}
is still completely open but, by answering precisely to our first question one should be able to construct counterexamples or to prove that, at least for the case of maximal relative irregularity, such examples cannot exist. It is worth to mention that, by the original work of Xiao (see \cite{Xia}), a non isotrivial fibration with maximal relative irregularity can exist only if $g\leq 7$. 
\vspace{2mm}

\noindent The paper is organized as follows. In Section \ref{SEC:1} we recall some definitions and facts about Prym varieties associated to ramified coverings and Prym maps that we are going to use extensively in what follows. In Section \ref{SEC:2} we will extend the techniques developed in \cite{Kan} for coverings with simple ramification to the case of arbitrary one. The main result is this theorem
\begin{thm*}[\ref{THM:MAINCOMP}]
With the notations of section \ref{SEC:2}, for any $\varphi\in \Sym^2(H^0(\omega_F)^{-})$ we have
\begin{equation}
\dPv(\varphi)=\sum_{j=1}^{n}\Res_{a_j}\left(\dfrac{m(\varphi)}{\pi^*\alpha}\right) dt_j+\left(\sum_{k=0}^n\dfrac{m(\varphi)}{\pi^*\alpha^2}(x_k)\right)ds.
\end{equation}
\end{thm*}

\noindent that describe the (dual of the) differential of the Period map in terms of residues of some meromorphic forms. In Section \ref{SEC:3}, given a covering $\pi:F\rightarrow E$ and assuming that $F$ is not hyperelliptic, we prove Theorem \ref{THM:HALFGEO}, a geometric criterion on the canonical model $F$ that is a sufficient condition in order to have $\dim(\Ker(\dP))=1$. Finally, in Section $\ref{SEC:4}$, we analyze in our framework the family that was constructed in \cite{Pir}. We will prove, using our framework, that the existence of the family is consistent with our Theorem as well as other interesting geometric aspects that may suggest a different way to approach, in the future, the problem of finding an answer to the second question by starting from the geometry of canonical models.
\vspace{2mm}

\noindent {\bf Acknowledgements.} The authors are both supported by INdAM - GNSAGA. The first author is supported by FIRB2012 "Moduli spaces and applications" and by IMUB during his stay at the University of Barcelona. The seminal ideas at the base of this article were born in the framework of the PRAGMATIC project of year 2016. We would like to thank Miguel {\'A}ngel Barja, V\'ictor Gonz\'alez-Alonso, Juan Carlos Naranjo and Gian Pietro Pirola for introducing us to the subject and for all the helpful remarks on the problem and the manuscript.


\section{Some preliminaries}
\label{SEC:1}

\noindent In this section we recall some definitions that we are going to use in the following sections.
\vspace{2mm}

\noindent Let $F,E$ be two smooth curves of genus $g\geq 2$ and $1$ respectively and consider the covering $\pi:F\rightarrow E$. One can consider the Albanese variety associated to $F$, which coincides with its Jacobian, because $F$ is a curve. Namely
\begin{equation}
J(F)=\dfrac{H^0(\omega_F)^{\vee}}{H_1(F,\Z)}=\Alb(F)=\dfrac{H^1(\OO_F)}{H^1(F,\Z)}.
\end{equation}
This is a principally polarized abelian variety of dimension $g$. As $E$ has genus $1$ we have $E=J(E)=\Alb(E)$. By the universal property of $\Alb(F)$ there is a map $\alb(\pi)$ such that the diagram 
$$\xymatrix{
F \ar@{->>}[r]^{\pi}\ar[d]\ar@{}[dr]|-{\circlearrowleft}& E\ar[d]^{=} \\
JF \ar@{->>}[r]_{\alb(\pi)}& JE
}$$
commutes, where the map $F\rightarrow JF$ is the Albanese map of $F$, also called the Abel-Jacobi map. The map $\alb(\pi)$ is also called the \textit{norm map of $\pi$}, $Nm_{\pi}$, and it is surjective. The \textit{generalized Prym variety} associated to $\pi:F\rightarrow E$ (or simply \textit{Prym variety}) is the connected component of $\Ker(\alb(\pi))$ that contains the $0$, i.e.
\begin{equation}
P(\pi)=\Ker(\alb(\pi))_0.
\end{equation}
$P(\pi)$ is an abelian variety of dimension $g-1$ with a natural polarization $\Theta_P$ given by $\Theta_{JF}|_P$ via the embedding 
$$
\xymatrix{P(\pi) \ar@{^{(}->}[r]& JF.}
$$
The map $\pi:F\rightarrow E$ induces a map $tr_{\pi}:H^0(\omega_F)\rightarrow H^0(\omega_E)$ called the \textit{trace} of $\pi$ (see Appendix A of \cite{Kan} for the definition). The trace satisfies $$tr_{\pi}\circ \pi^*=\Deg(\pi)\Id_{H^0(\omega_E)}.$$
If we define 
\begin{equation}
H^0(\omega_F)^-=\Ker(tr_{\pi})
\end{equation}
we have a canonical splitting 
\begin{equation}
H^0(\omega_F)=\pi^*H^0(\omega_E)\oplus H^0(\omega_F)^-
\end{equation}
and we can identify the quotient $H^0(\omega_F)/\pi^*H^0(\omega_E)$ with $H^0(\omega_F)^-$. In particular, the tangent bundle of $P(\pi)$ can be described as
\begin{equation}
T P(\pi) = \left(\dfrac{H^0(\omega_F)}{\pi^*H^0(\omega_E)}\right)^{\vee}\otimes \OO_{P(\pi)} = (H^0(\omega_F)^-)^{\vee}\otimes \OO_{P(\pi)}. 
\end{equation}
\vspace{2mm}

\noindent Now we will introduce the families of coverings of elliptic curves we are interested in. Fix a smooth curve $F$ of genus $g\geq 2$ and consider a degree $d$ covering $\pi:F\rightarrow E$, where $E$ is an elliptic curve. Denote with $$R=\sum_{j=1}^n (n_j-1)a_j$$ the ramification divisor and call $b_j$ the branch point corresponding to the ramification point $a_j$, i.e. $\pi(a_j)=b_j$. Thus $n_j$ is the degree of $\pi$ when restricted to a suitable neighborhood of $a_j$. 
\vspace{2mm}

\noindent Fix a generator $\alpha$ of $H^0(\omega_E)$. Choose a suitable set $\{\Delta_j\}$ of coordinate neighborhoods centered in the points $b_j$ and call $w_j$ the corresponding coordinate on $E$.  This is not needed at the moment but observe that we can assume that $\alpha|_{\Delta_j}=dw_j$. We can chose a collection of pairwise disjoint coordinate neighborhoods $(U_j,z_j)$ centered in $a_j$ in such a way that $w_j=\pi|_{U_j}(z_j)=z_j^{n_j}$.
\vspace{2mm}

\noindent Denote by $\caH_E$ the polydisc 
$\Pi_{j=1}^{n}\Delta_j$ and consider the coordinates $t=(t_j)_{j=1}^n$  defined by the relation
$$t_j(P_1,\cdots, P_n)=w_j(P_j).$$
We can consider, as in Section 4.1 of \cite{Kan}, a family 
$$(\Psi,f):\caF\rightarrow E\times \caH_E$$
of $d$-sheeted branched coverings deforming $\pi$ parametrized by 
$\mathcal{H}_E$ such that
\begin{equation}
\label{EQ:LOCEXPROFPSI}
w_j=\Psi|_{U_j}(z_j,t)=z_j^{n_j}+t_j.
\end{equation}
In this way, to each $b'\in \cH_E$, it is associated a covering $\pi_{b'}:F_{b'}\rightarrow E$ which is a deformation of $\pi$, the central fiber. Note that (\ref{EQ:LOCEXPROFPSI}) forces the ramification orders to remain costant and allows different branch points to move indipendently. This is what we will call in the following {\it local family of coverings over $E$ with central fiber $\pi$ parametrized by $\cH_E$}.

\noindent The tangent space to $\cH_E$ in $b=(b_1,\dots,b_n)\in \cH_E$ is
$$T_{b}\cH_E\simeq \bigoplus_{j=1}^{n}T_{b_j}E\simeq  \bigoplus_{j=1}^{n}\C\pder{t_j},$$
where the tangent vectors on the right are evaluated in $0$. 
\vspace{2mm}

\noindent We can also take into account the deformation of the elliptic curve. Indeed, following \cite{ACG}, if one chooses $c\in E$ not among the $b_j$ and considers a small coordinate neighborhood $(N,v)$ of $c$ (eventually shrinking $\Delta_j$ in such a way that for all $j$ they are disjoint from $N$), one can consider the associated Schiffer variation $\m{E}\rightarrow N$ of $E$ with coordinate $s$. Observe that we can assume $\alpha|_{N}=dv$. Taking into account also the movement of the branch points one has a family $f:\caF\rightarrow \cH_E\times N$ of curves of genus $g$ that fits into the diagram
\begin{equation}
\label{EQ:LFC}
\xymatrix{
\caF\ar[d]_{f}\ar[r]^{p} & \caE\ar[d]\\
\cH_E\times N\ar[r] & N 
}
\end{equation}
For a choice $(b',s')\in \cH=\cH_E\times N$ we have an elliptic curve $E_{s'}$, the fiber of the map $\caE\rightarrow N$ over $s'$, a curve $F_{(b',s')}$ of genus $g$ and a covering $$\pi_{(b',s')}=p|_{F_{(b',s')}}:F_{(b',s')}\rightarrow E_{s'}.$$
For this reason, the map $p$ is what we will call {\it local family of coverings with central fiber $\pi$ parametrized by $\cH$} or, simply, {\it local family of coverings}.
The tangent space to $\cH$ in $(b,s)$ is 
$$T_{(b,s)}\cH\simeq  \left(\bigoplus_{j=1}^{n}\C\pder{t_j}\right)\oplus \C\pder{s}.$$
and, clearly, containts $T_{b}\cH_E$ in a natural way. We stress that, through the whole article, unless otherwise stated, we will always refer to the families of coverings constructed in this sections. 
\vspace{2mm}

\noindent If we have a family of coverings parametrized by $\cH$, for each $(b,s)$ we can construct the Prym variety associated to the covering. Moreover, the type of polarization remains constant. Hence we can consider the \textit{Prym map}
\begin{equation}
\xymatrix@R=5pt@C=30pt{
\cH \ar[r]^-{\Phi} & \A_{g-1} \\
(b,s) 
\ar@{|->}[r] & [P(\pi_{(b,s)})]
}
\end{equation}
where $\A_{g-1}$ is the moduli space of abelian varieties with polarization (which will be omitted) equal to the one of the central fiber. In the same way one has the Prym map $\Phi_E$ associated to a local family of coverings over $E$.
\vspace{2mm}

\noindent To avoid technical subtleties around singular points of $\A_{g-1}$, we will consider the period map $P:\cH\rightarrow \mD$ (or $P_E:\cH_E\rightarrow \mD$) instead of the Prym map $\Phi$ (respectively $\Phi_E$), where $\mD$ is a suitable period domain for $\A_{g-1}$. The interested reader is referred to \cite[Section 3]{Kan} for technical details.
\vspace{2mm}

\noindent Through the whole article, giving two sections $s_1,s_2\in H^0(\OO_X(D))$ we will write $s_1\spr s_2$ to mean their symmetric product, i.e.
$$\frac{1}{2}(s_1\otimes s_2+s_2\otimes s_1)\in \Sym^2(H^0(\OO_X(D))).$$
If $s_i\in H^0(\OO_X(D_i))$, $s_1\cdot s_2$ will mean the evaluation of $s_1\otimes s_2$ in $H^0(\OO_X(D_1+D_2))$ under the multiplication map.


\section{A direct formula for the codifferential of the Prym map}
\label{SEC:2}

\noindent In this section we will prove an explicit formula for the codifferential of the period map in terms of the residue at the ramification points of some forms. The framework is similar to the one in \cite{Kan} with the main difference being that we don't restrict ourselves to the case of simple ramification. First of all we introduce some notations.
\vspace{2mm}

\noindent Fix an elliptic curve $E$ and let $\pi:F\rightarrow E$ be a covering of $E$ with $F$ of genus $g$. Consider  
$$(\Psi,f):\caF\rightarrow E\times \caH_E,$$
the local family of coverings with fixed base $E$, central fiber $\pi$ and parameter space $\cH_E$ constructed in Section \ref{SEC:1}. By construction, it induces a family $f:\caF\rightarrow \cH_E$ with central fiber $\caF_o=F$. If we consider a minimal versal deformation $f':\caF'\rightarrow M$ of $F$ then the previous family is induced by $f'$ by means of a pullback. More precisely there exists a holomorphic map
$h_E:\cH_E\rightarrow M$ such that 
\begin{equation}
\label{EQ:COMMDIAGEFIX}
\xymatrix{
\caF\ar[d]_{f}\ar[r] & \caF'\ar[d]^{f'}\\
\cH_E\ar[r]_{h_E} & M 
}
\end{equation}
is commutative. Being $f'$ a minimal versal deformation we have $$T_oM\simeq H^1(T_F)\simeq H^0(\omega_F^{\otimes 2})^{\vee}.$$ Moreover, under this identification, if we take a tangent vector $v$ in $T_o\cH_E$ and evaluate $dh_E$ in $v$ we get the Kodaira-Spencer map $KS_E$ associated to $\caF\rightarrow \cH_E$ evaluated in $v$.
\vspace{2mm}

\noindent We are able to prove the first important part of Theorem \ref{THM:MAINCOMP}. 

\begin{prop}
\label{PROP:DH1}
Using the identifications introduced above, we have that
$$dh_E^{\vee}:T_{o}^{\vee}M\rightarrow T_{o}^{\vee}\cH_E$$
can be written as
\begin{equation}
dh_E^{\vee}(\varphi)=\sum_{j=1}^{n}\gamma_j dt_j\qquad\mbox{where }\qquad  \gamma_j=2\pi i \Res_{a_j}\left(\dfrac{\varphi}{\pi^*\alpha}\right)
\end{equation}
and $\varphi\in T_o^{\vee}M=H^0(\omega_F^{\otimes 2})$.
\end{prop}

\begin{proof}
For every $\varphi\in H^0(\omega_F^{\otimes 2})$ we have that $dh_E^{\vee}(\varphi)$ is identified, as cotangent vector on $M$ in $o$, by the complex numbers $\gamma_j$ such that
$$dh_E^{\vee}(\varphi)=\sum_{j=1}^{n}\gamma_j dt_j.$$
By construction, we can obtain these numbers simply by pairing $dh_E^{\vee}(\varphi)$ against $\pder{t_j}$:
$$\gamma_j=dh_E^{\vee}(\varphi)\left(\pder{t_j}\right)=\varphi\left(dh_E\left(\pder{t_j}\right)\right)=\varphi\left(KS_E\left(\pder{t_j}\right)\right).$$
In order to develop the computation we may proceed using a description of $KS_E$ in terms of the \v{C}ech cohomology (details of this can be found in \cite{Hor1}). To do it consider the exact sequence
\begin{equation}
\label{EXSEQ:STANDARD}
\xymatrix{0\ar[r] & T_F \ar[r]^{d\pi} & \pi^*T_E\ar[r]^{\psi} & \m{R} \ar[r] & 0}
\end{equation}
and let $\delta$ be the coboundary map $H^0(\m{R})\rightarrow H^1(T_F)$. Then $KS_E$ factors as $\delta\circ\tau=KS_E$ where $\tau:T_bH\rightarrow H^0(\m{R})$ is the {\it characteristic map} of the family (see \cite{Hor1} for the definition and the proof of this fact). Hence we can unfold the calculation using these exact sequences. 
\vspace{2mm}

\noindent If one restricts the exact sequence (\ref{EXSEQ:STANDARD}) on $U_j$ (or some sufficiently small subset of this coordinate neighborhood), it can be identified with
\begin{equation}
\xymatrix{0\ar[r] & \OO_{U_j}\pder{z_j} \ar[r]^{d\pi} & \OO_{U_j}\pder{w_j}\ar[r]^{\psi} & \m{R}|_{U_j} \ar[r] & 0.}
\end{equation}
The first map sends $\pder{z_j}$ to $n_jz_j^{n_j-1}\pder{w_j}$ while the second one is simply the restriction to the ramification locus. Let $\m{U}=\{U_0,U_1,\dots,U_n\}$ where $U_j$ for $j=1,\dots,n$ are the neighborhoods defined above and $U_0=F\setminus\{ a_{j}\}$. Let, as usual, $U_{\alpha,\beta}$, be a shorthand for $U_{\alpha}\cap U_{\beta}$ with $\alpha<\beta$. If $\eta=[\eta_j]\in H^0(\m{U},\m{R})$ with $\eta_0=0$ and $\eta_j=p_j(z_j)\pder{w_j}$ we have 
$$\delta\left(\eta\right)=\left[\lambda_{\alpha,\beta}\right]\qquad \mbox{with }\qquad  \lambda_{0,j}=\frac{p_j(z_j)}{n_jz^{n_j-1}}\pder{z_j}$$
for $j>0$ and $\lambda_{\alpha,\beta}=0$ if $\alpha,\beta>0$. Following \cite{Hor1} and using Equation (\ref{EQ:LOCEXPROFPSI}) we have
\begin{equation}
\tau\left(\pder{t_j}\right)=[\tau_k^{(j)}]\qquad \mbox{with }\qquad  \tau_k^{(j)}=\begin{cases} 
      0 & k\neq j \\
      \pder{w_j} & k=j. 
   \end{cases}
\end{equation}
Hence we have 
$$KS_E\left(\pder{t_j}\right)=\delta\left(\tau\left(\pder{t_j}\right)\right)=[\chi_{\alpha,\beta}^{(j)}]\qquad \mbox{with }\qquad  \chi_{\alpha,\beta}^{(j)}=\begin{cases} 
      \frac{1}{n_jz^{n_j-1}}\pder{z_j} & (\alpha,\beta)=(0,j) \\
      0 & \mbox{otherwise}. 
   \end{cases}$$
If $\varphi\in H^0(\omega_F^{\otimes 2})$ we can represent it as \v{C}ech-cocycle as $[\phi_j]$ where 
$$\phi_0=\phi|_{U_0}\quad \mbox{ and }\quad \phi_j=q_j(z_j)dz_j^2$$
are the local expressions of $\varphi$ in coordinates around $a_j$. The numbers we are interested in are simply the ones obtained by considering the perfect pairing
\begin{equation}
\label{EQ:PERFPAIR}
\xymatrix{
H^0(\omega_F^{\otimes 2})\otimes H^1(T_F)\ar[r] & H^1(\omega_F)\ar[r]^-{\simeq} & \C
}\end{equation}
applied to $KS_E\left(\pder{t_j}\right)$ and $\varphi$. Using \v{C}ech cohomology, the image in $H^{1}(\omega_F)$ of our product is given by
the \v{C}ech class $[\epsilon_{\alpha,\beta}^{(j)}]$ with
$$
\epsilon_{\alpha,\beta}^{(j)}=\begin{cases} 
      \frac{q_j(z_j)}{n_jz^{n_j-1}}dz_j & (\alpha,\beta)=(0,j) \\
      0 & \mbox{otherwise}. 
\end{cases}
$$
What remains to be proven is the analogous to the calculation of \cite{Kan} for the case of simple ramification: roughly, one can adapt the techniques of \cite[pag. 14-15]{ACG} to develop the last isomorphism of (\ref{EQ:PERFPAIR}) in order to finally get
$$\gamma_j=2\pi i \Res_{0}\dfrac{q_j(z_j)dz_j^2}{n_jz_j^{n_j-1}dz_j}=2\pi i \Res_{a_j}\dfrac{\varphi}{\pi^*\alpha}.$$
\end{proof}

\noindent Consider now the family $p:\caF\rightarrow \caE$ with central fiber $\pi:F\rightarrow E$ and parameter space $\cH=\cH_E\times N$ as defined in Section \ref{SEC:1}.
As before, we have an induced deformation $f:\caF\rightarrow \cH$ of $F$, its associated Kodaira-Spencer map $KS$ and, when a minimal versal deformation $f':\caF'\rightarrow M$ of $F$ is chosen, an holomorphic map $h:\cH\rightarrow M$ such that
\begin{equation}
\label{EQ:COMMDIAGENONFIX}
\xymatrix{
\caF\ar[d]_{f}\ar[r] & \caF'\ar[d]^{f'}\\
\cH\ar[r]_{h} & M 
}
\end{equation}
is commutative. Again, as $T_oM\simeq H^1(T_F)$, we can identify $dh$ with $KS$. We will denote by $x_1,\dots,x_d$ the points of the fiber of $\pi$ over the point $c$ which, by construction, are all different.

\begin{prop}
\label{PROP:DH2}
Using the identifications introduced above, we have that
$$dh^{\vee}:T_o^{\vee}M\rightarrow T_{o}^{\vee}\cH$$
can be written for any $\varphi\in H^0(\omega_F^{\otimes 2})=T_o^{\vee}M$ as $dh^{\vee}(\varphi)=\sum_{j=1}^{n}\gamma_j dt_j+\gamma ds$ where
\begin{equation}
\gamma_j=2\pi i \Res_{a_j}\left(\dfrac{\varphi}{\pi^*\alpha}\right)\quad\mbox{and}\quad\gamma=2\pi i\sum_{k=1}^{d}\dfrac{\varphi}{\pi^*\alpha}(x_k).
\end{equation}
\end{prop}

\begin{proof}
As before, by duality,
$$dh^{\vee}(\varphi)=\varphi\circ dh=\varphi\circ KS.$$
It is then clear that the formula for $\gamma_j$ follows directly from Proposition \ref{PROP:DH1}. The one that gives $\gamma$, as it involves calculations done far from the ramification points, doesn't depend on the type of the ramifications. Hence, the one given in \cite{Kan} when $\pi$ as only simple ramification is still valid.
\end{proof}

\noindent Recall that we have a decomposition of $H^0(\omega_F)$ given by $H^0(\omega_F)^{-}\oplus \pi^{*}H^0(\omega_E)$ where the first space is the vector space of $1$-forms on $F$ with trivial trace. This induces a decomposition on $\Sym^2(H^0(\omega_F))$. Unless otherwise specified, consider $\Sym^2(H^0(\omega_F)^-)$ as a subspace of $\Sym^2(H^0(\omega_F))$ in the natural way. Let
$m:\Sym^2(H^0(\omega_F)^2)\rightarrow H^0(\omega_F^{\otimes 2})$ be the multiplication map. Denote by $P:\cH\rightarrow \mD$ the period map associated to the Prym map $\Phi:\cH\rightarrow \A_{g-1}$ where $\mD$ is a suitable period domain. We are ready to prove Theorem \ref{THM:MAINCOMP}.

\begin{thm}
\label{THM:MAINCOMP}
With the notation introduced in this section, for any $\varphi\in \Sym^2(H^0(\omega_F)^{-})$ we have
\begin{equation}
\dPv(\varphi)=\sum_{j=1}^{n}\Res_{a_j}\left(\dfrac{m(\varphi)}{\pi^*\alpha}\right) dt_j+\left(\sum_{k=1}^d\dfrac{m(\varphi)}{\pi^*\alpha^2}(x_k)\right)ds.
\end{equation}
\end{thm}

\begin{proof}
Theorem 3.21 of \cite{Kan} expresses the codifferential of the period map calculated in $\varphi\in \Sym^2(H^0(\omega_F)^-)$ and paired with $\pder{t_j}$ as
$$\varphi\left(KS\left(\pder{t_j}\right)\right)$$
without any restriction on the ramification type. In particular, this formula, together with Proposition \ref{PROP:DH2} ends the proof of the Theorem. 
\end{proof}

\begin{rmk}
\label{REM:SPLI}
As a consequence of the last Theorem we can conclude that, if we fix $E$, the codifferential $\dPEv:\Sym^2(H^0(\omega_F)^-)\rightarrow T^{\vee}_o\cH_E$ factors as
\begin{equation}
\label{DIAG:COTANG}
\xymatrix{
H^0(\omega_F^{\otimes 2})\ar[d]_{dh_E^{\vee}} & \Sym^2\left(H^0(\omega_F)\right)\ar[l]_-{dT^{\vee}} & \\
T^{\vee}_{o}\cH_E &\Sym^2(H^0(\omega_F)^-)\ar@{^{(}->}[u]_-{\sigma}\ar[l]^-{\dPEv}
}
\end{equation}
where $T$ is the Torelli map (so that $m=dT^{\vee}$) and $\sigma$ is the lifting of the projection
of $\Sym^2(H^0(\omega_F))\rightarrow\Sym^2(H^0(\omega_F)^-)$ induced by the decomposition $H^0(\omega_F)=H^0(\omega_F)^{-}\oplus \pi^{*}H^0(\omega_E)$. The commutativity of the diagram is a consequence of Proposition \ref{PROP:DH1} as, for any $\varphi\in H^0(\omega_F)\spr \pi^*H^0(\omega_E)$, we have that $\varphi/\pi^*\alpha$ is holomorphic and hence has residue zero everywhere.
\end{rmk}


\section{A geometric approach via the canonical embedding}
\label{SEC:3}

\noindent In this section we will use the technical result of the previous section in order to prove that $\dim(\Ker(\dPE))=1$ for arbitrary ramification types and a geometric criterion to determine whether $\dim(\Ker(\dP))=1$ or not. First we fix some notation and facts about the canonical curves that we are going to use extensively in the following.
\vspace{2mm}

\noindent As $F$ has genus $g\geq 3$ and is not hyperelliptic, we may identify it with its canonical model in $\PP=\PP H^0(\omega_F)^{\vee}$. This is a non-degenerate curve of degree $2g-2$, which is also projectively normal by a classical result of Max Noether (see, for example, \cite{ACGH}). One of the consequences of this fact is that the multiplication map $m_k:\Sym^kH^0(\omega_F)\rightarrow H^0(\omega_F^{\otimes k})$ is surjective. As before we will denote $m_2$ simply by $m$. We will use frequently the natural identifications 
$H^0(\OO_{\PP}(d))=\Sym^d H^0(\omega_F)$ which enable us to identify 
$\PP(\Ker(m_d))$ with the space of hypersurfaces of degree $d$ in $\PP H^0(\omega_F)^{\vee}$ that contain $F$. By abuse of notation we will simply say that an element in $\Sym^d H^0(\omega_F)$ is an hypersurface of degree $d$ if no confusion arises.
In particular, if $I_F$ is the ideal sheaf of $F$ in $\PP H^0(\omega_F)^{\vee}$, then $\Ker(m)=H^0(I_F(2))$ gives the set of all quadrics in $\PP H^0(\omega_F)^{\vee}$ containing the curve $F$, and has dimension $\frac{(g-2)(g-3)}{2}$.
\vspace{2mm}

\noindent Recall that the decomposition $$H^0(\omega_F)=H^0(\omega_F)^{-}\oplus \pi^*H^0(\omega_E)$$
where the first space is the space of forms with zero trace.
\vspace{2mm}

\noindent Since elements in $H^0(\omega_F)$ are linear equations on $\PP H^0(\omega_F)^{\vee}$, all the hyperplanes defined by elements in $H^0(\omega_F)^{-}$  intersect in a single point $q^-$ of $\PP$ which is a point really important in what will follows. We have also a particular hyperplane, the one defined by the subspace $\pi^*H^0(\omega_E)$ which will be denoted by $H^-$. 
More precisely, 
$$q^-=\PP((H^0(\omega_F)^-)^{\perp})\qquad \mbox{ and }\qquad  H^-=\PP((\pi^*H^0(\omega_E))^{\perp})$$
As before, we will fix a generator $\alpha$ of $H^0(\omega_E)$ so that
\begin{equation}
\label{EQ:DECQUADRICS}
\Sym^2(H^0(\omega_F))=\Sym^2(H^0(\omega_F)^-)\oplus \left(\pi^*\alpha\spr H^0(\omega_F)\right).
\end{equation}
Given a quadric $Q$ in $\PP$ we will denote by $G_Q\in \Sym^2(H^0(\omega_F))$ one of its equations and by $G_Q^-\in \Sym^2(H^0(\omega_F)^-)$ and $\omega_Q\in H^0(\omega_F)$ the only elements such that $$G_Q=G_Q^-+\pi^*\alpha\spr \omega_Q$$ under the decomposition (\ref{EQ:DECQUADRICS}). Finally, given a quadric $Q$, we will denote by $Q^{-}$ the cone given by the equation $G_Q^-$, i.e. the quadric such that $G_{Q^{-}}=G_{Q^{-}}^{-}=G_Q^{-}$.
\vspace{2mm}

\noindent In order to prove Theorem \ref{THM:KERDPHIE} we will need the following result:

\begin{Lem}
\label{LEM:QUAKER}
We have a natural inclusion of $H^0(I_F(2))$ in $\Ker(\dPEv)$.
\end{Lem}

\begin{proof}
Recall that, fixed a family of coverings with base $E$ and central fiber $\pi:F\rightarrow E$, by fixing a minimal versal deformation $\caF'\rightarrow M$ of $F$, we can construct $h_E:\caH\rightarrow M$ like in diagram (\ref{EQ:COMMDIAGEFIX}). As observed in Remark \ref{REM:SPLI} we have a commutative diagram

\begin{equation}
\label{EQ:DIAG1}
\xymatrix{
 &  & 0 & 
\\
0\ar[r] &\Ker \dPEv\ar@{^{(}->}[r]^-{j} & \Sym^2(H^0(\omega_F)^-)\ar[u] \ar[r]^-{d_oP_E^{\vee}}  & T^{\vee}_{o}\cH_E 
\\
0\ar[r] & H^0(I_F(2))\ar@{^{(}-->}[u]^-{\gamma}\ar@{^{(}->}[r]^-{\iota} & 
\Sym^2(H^0(\omega_F))\ar@{->>}[u]^-{pr} \ar@{->>}[r]^-{m} & 
H^0(\omega_F^{\otimes 2})\ar[u]_{dh_E^{\vee}}\ar[r] & 0
\\
 & & H^0(\omega_F)\spr \pi^*H^0(\omega_E) \ar@{^{(}->}[u] 
\\
 & & 0\ar[u] &
}
\end{equation}

\noindent It is easy to see that the image of $pr\circ \iota$ lives in $\Ker(\dPEv)$ so we have a well defined map $\gamma:H^0(I_F(2))\rightarrow \Ker(\dPEv)$. We want to prove that this map is indeed injective. This follows from the geometry of the problem. Indeed, if a quadric $Q$ contains $F$, i.e. if the quadric has equation $$G_Q=G_Q^{-}+\pi^{*}\alpha\spr\omega_Q\in H^0(I_F(2)),$$ and if $\gamma(G_Q)=0$ then we have that the quadric has equation $\pi^{*}\alpha\spr\omega_Q$. But this is impossible because such a quadric the union of two planes (one of which is $H^-$) and the canonical curve is non-degenerate. Hence $\gamma$ is injective.
\end{proof}

\begin{thm}
\label{THM:KERDPHIE}
Let $\pi:F\rightarrow E$ be a covering with $F$ non-hyperelliptic, consider a local family of coverings with base $E$ and parameter space $\cH_E$ constructed in Section \ref{SEC:1}. Let $P_E$ be the period mapping associated to the Prym map $\Phi_E$. Then $\dim(\Ker(\dPE))=1$.
\end{thm}

\begin{proof}
First of all, observe that for dimensional reasons, one has $\dim(\Ker(\dPE))=1$ if and only if 
$$\dim(\Ker(\dPEv))=\frac{g(g-1)}{2}-n+1.$$
From the splitting $H^0(\omega_F)=H^0(\omega_F)^{-}\oplus \pi^*H^0(\omega_E)$ we have the commutative diagram
$$
\xymatrix{
0\ar[r] & H^0(\omega_F)\ar[r]^-{\spr\pi^*\alpha} & \Sym^2(H^0(\omega_F))\ar[r]^{pr}\ar[rd]_-{\Psi} & \Sym^2(H^0(\omega_F)^-)\ar[r]\ar[d]^{d_oP_E^{\vee}} & 0 
\\
 & & & T_o\cH_E 
}
$$
  with $\Psi$ defined by extending the formula in Theorem \ref{THM:MAINCOMP} to $\Sym^2(H^0(\omega_F))$. This can be done because, as previously observed (see Remark \ref{REM:SPLI}), $\dPEv(H^0(\omega_F)\spr\pi^*\alpha)= \{0\}$.
In particular, we have the relation
\begin{equation}
\dim(\Ker(\dPEv))=\dim(\Ker(\Psi))-\dim(\Ker(pr))=\dim(\Ker(\Psi))-g.
\end{equation}
By definition, $\Psi$ factors through the multiplication map $m$ as $\Psi=\bar{\Psi}\circ m$. The map $\bar{\Psi}$ is well defined as, by Lemma \ref{LEM:QUAKER}, $\Ker(m)\subset \Ker(\Psi)$. 
$$
\xymatrix{
0\ar[r] & H^0(\omega_F)\ar[r]^-{\spr\pi^*\alpha} & \Sym^2(H^0(\omega_F))\ar[r]^{pr}\ar[rd]^{\Psi}\ar[d]^{m} & \Sym^2(H^0(\omega_F)^-)\ar[r]\ar[d]^{d_oP_E^{\vee}} & 0 
\\
 & & H^0(\omega_F^{\otimes 2})\ar[r]^{\bar{\Psi}}& T_o\cH_E 
}
$$

\noindent Being $m$ surjective (as $F$ is non-hyperelliptic) we obtain the further relation
\begin{equation}
\dim(\Ker(\Psi))=\dim(\Ker(\bar{\Psi}))+\dim(\Ker(m))=\dim(\Ker(\bar{\Psi}))+\frac{(g-2)(g-3)}{2}.
\end{equation}
As the divisor associated to $\pi^*\alpha$ is exactly $R$, the ramification divisor, we have that $\omega_F=\OO_F(R)$ and there is an exact sequence
\begin{equation}
\xymatrix{
0\ar[r] & \omega_F\ar[r]^-{\cdot \pi^*\alpha} & \omega_F^{\otimes 2}\ar[r] & \omega_F^{\otimes 2}|_R\ar[r] & 0
}
\end{equation}
which yields, denoting with $V$ the quotient $H^0(\omega_F^{\otimes 2})/(H^0(\omega_F)\cdot\pi^*\alpha)$, the exact sequences
\begin{equation}
\label{EQ:TWOEXSEQ}
\xymatrix@R=10pt{
0\ar[r] & H^0(\omega_F)\ar[r]^-{\cdot \pi^*\alpha} & H^0(\omega_F^{\otimes 2})\ar[r]^-{\epsilon} & V\ar[r] & 0 \\
0\ar[r] & V\ar[r]^-{\zeta} & H^0(\omega_F^{\otimes 2}|_R)\ar[r] & H^1(\omega_F)\ar[r] & 0
}
\end{equation}
Let $\eta\in \Ker(\epsilon)$. We want to prove that $\bar{\Psi}(\eta)=0$. This is easily proven: write $\eta$ as $\omega\cdot\pi^*\alpha$ and observe that
$$\bar{\Psi}(\eta)=(\bar{\Psi}\circ m)(\omega\spr\pi^*\alpha)=(\dPEv\circ pr)(\omega\spr\pi^*\alpha)=0$$
because $\omega\spr\pi^*\alpha\in \Ker(pr)$. In particular, $\Ker(\epsilon)\subset \Ker(m)$ and we can define a map $\lambda:V\rightarrow T_{o}^{\vee}\cH_E$ such that $\bar{\Psi}=\lambda\circ\epsilon$. 
Moreover
\begin{equation}
\dim(\Ker(\bar{\Psi}))=\dim(\Ker(\lambda))+g.
\end{equation}
Using the second exact sequence in \ref{EQ:TWOEXSEQ} we can also define a map $\mu:H^0(\omega_F^{\otimes 2}|_R)\rightarrow T_{o}\cH_E$ such that $\mu\circ \zeta=\lambda$.  
\begin{equation}
\xymatrix{
	&
	H^0(I_F(2))\ar@{^{(}->}[d]
\\
    H^0(\omega_F)\ar@{^{(}->}[r]^-{\spr\pi^*\alpha}
    	\ar@{^{(}->}[rd]^{\cdot \pi^*\alpha} & 
    \Sym^2(H^0(\omega_F))\ar@{->>}[rr]^-{pr} \ar[rrd]^{\Psi}
    	\ar@{->>}[d]^{m} &
    &
    \Sym^2(H^0(\omega_F)^-)\ar[d]^{d_oP_E^{\vee}}
\\
	&
	H^0(\omega_F^{\otimes 2})\ar[rr]^-{\bar{\Psi}}
		\ar@{->>}[rd]_-{\epsilon}
		\ar[rrd]|!{[rd];[rr]}\hole & 
    &
    T_{o}^{\vee}\cH_E
\\
	&
	&
    V\ar[ur]^{\lambda}\ar@{_{(}->}[r]_-{\zeta} & 
    H^0(\omega_F^{\otimes 2}|_R)\ar[u]_{\mu} \ar@{->>}[r] & 		
    H^1(\omega_F)
\\
}
\end{equation}
Note that we have several ways to define $\mu$. Since $\omega_F^{\otimes 2}|_R=\omega_F(R)|_R$ the global sections of $\omega_F^{\otimes 2}|_R$ are just collections of meromorphic tails on the points of ramification, i.e. elements 
$$\left\{\sum_{j=1}^{n_k-1} \beta_{jk}\dfrac{dz_k}{z_k^j}\right\}_{a_k\in R}$$ 
where $n_k$ is the ramification index of the point $a_k$. In particular, we can define $\mu$ as the map which gives the residue in the corresponding point of the meromorphic tail. This ensures that the diagram is commutative.
In addition, $\mu$ is surjective (this because the image of a collection of meromorphic tails $\{s_k\}$, one for each point of ramification, with $\beta_{1m}=\delta_{km}$, generates the image), and as a consequence, $\zeta|_{\Ker(\lambda)}$ is an isomorphism between $\Ker(\lambda)$ and $\Ker(\mu)$.
Hence,
\begin{multline}
\label{EQ:DIMPED}
\dim\Ker(\dPEv)=\dim(\Psi)-g=\dim(\bar{\Psi})+\dim(\Ker(m))-g=\\
=\dim(\lambda)+\dim(\Ker(m))=\dim(\mu)+\dim(\Ker(m))=\\
=\h^0(\omega_F(R)|_R)-\dim T_{o}^{\vee}\cH_E+\dim(\Ker(m))=\frac{g(g-1)}{2}-n+1
\end{multline}
as wanted.
\end{proof}

\noindent Now we will prove the first main theorem:

\begin{thm}
\label{THM:RESFORMULA}
Let $\pi:F\rightarrow E$ be a covering with $F$ non-hyperelliptic, consider the local family of coverings with parameter space $\cH$ constructed in Section \ref{SEC:1}. Let $P$ be the period mapping associated to the Prym map $\Phi:\caH\rightarrow \A_{g-1}$. Using the same notations of Theorem \ref{THM:MAINCOMP} we have 
$$\dim(\Ker(\dP)))=1 \qquad\Longleftrightarrow\qquad \exists \beta\in\Ker(\dPEv) \quad|\quad \sum_{k=1}^d\dfrac{m(\beta)}{\pi^*\alpha^2}(x_k)\neq 0.$$
\end{thm}

\begin{proof}
First of all consider the diagrams
$$
\xymatrix@C=30pt{
T_{o}\cH \ar[r]^-{d_oP} & T_{P(o)}\mD\\
T_{o}\cH_E \ar[ru]_-{d_oP_E}\ar@{^{(}->}[u]
}
\qquad
\xymatrix@C=30pt{
T^{\vee}_{o}\cH\ar@{->>}[d]  & T^{\vee}_{P(o)}\mD \ar[l]_-{d_oP^{\vee}} \ar[dl]^-{d_oP_E^{\vee}}\\
T^{\vee}_{o}\cH_E 
}
$$
and observe that one always has
$$\Ker(\dPE)\subseteq\Ker(\dP)\qquad \Ker(\dPv)\subseteq\Ker(\dPEv).$$
Moreover, the codimensions are at most $1$. If one considers the exact sequences
$$
\xymatrix@R=5pt{
0 \ar[r] & \Ker(\dPE) \ar[r]& T_o\cH_E \ar[r]& T_{P(o)}\mD \ar[r]& \Ker(\dPEv)^{\vee} \ar[r]& 0\\
0 \ar[r] & \Ker(\dP) \ar[r]& T_o\cH \ar[r]& T_{P(o)}\mD \ar[r]& \Ker(\dPv)^{\vee} \ar[r]& 0
}
$$
it is clear that $\Ker(\dPE)=\Ker(\dP)$ if and only if $\Ker(\dPv)\subsetneq\Ker(\dPEv)$. Hence we have
$$\dim(\Ker(\dP)))=1 \Longleftrightarrow \Ker(\dPv)\subsetneq\Ker(\dPEv).$$
This is true if and only there exists an element $\beta\in\Ker(\dPEv)$ on which $\dPv$ doesn't vanish. This can only be possible if $\dPv(\beta)$ is not zero on $\pder{s}$, where $s$ is the parameter taking into account the moduli of the elliptic curve.
By using Theorem \ref{THM:MAINCOMP} we have $$\dPv(\beta)=\sum_{k=1}^d\dfrac{m(\beta)}{\pi^*\alpha^2}(x_k)$$
and this concludes the proof.
\end{proof}

\noindent This result improves the one in \cite{Kan} where it is proved only for simple ramification. In the same work is proved that, for simple ramification, having the sum in Theorem \ref{THM:RESFORMULA} different from zero for some $\beta\in \Ker(\dPEv)$ is equivalent to ask that the intersection of the quadrics that contain the canonical model of $F$ doesn't contain the point $q^-$ defined before. Unfortunately, in the case of arbitrary ramification, we are not able to prove this equivalence but only one implication.

\begin{thm}
\label{THM:HALFGEO}
With the same hypotesis of Theorem \ref{THM:RESFORMULA}, if we identify $F$ with its canonical model in $\PP H^0(\omega_F)^{\vee}$, then we have
\begin{equation}
\label{EQ:HALFGEO}
q^-\not\in \bigcap_{F\subset Q} Q\Longrightarrow \dim(\Ker(\dP))=1,
\end{equation}
where $Q$ ranges in the set of quadrics of $\PP H^0(\omega_F)^{\vee}$ containing $F$.
\end{thm}

\noindent The proof of the theorem uses some arguments developed in \cite{Kan} that we have summarized in the following Lemma.

\begin{Lem}
\label{LEM:FUNCTPOINT}
Let $Q$ be a quadric of $\PP H^0(\omega_F)^{\vee}$ containing $F$ and denote by $G_Q=G_Q^-+\pi^*\alpha\spr\omega_Q$ one of its equations. Then
$$\sum_{k=1}^d\dfrac{m(G_Q^{-})}{\pi^*\alpha^2}(x_k)=0 \Longleftrightarrow G_Q(q^-)=0\Longleftrightarrow q^- \in Q.$$
\end{Lem}

\begin{proof}
The last statement is clear by definition so we really need to prove only the first one. First of all observe that we can choose the coordinate $s$ in such a way that $\alpha$ is locally given by $ds$.
Then, as $G_Q^{-}=G_Q-\pi^*\alpha\spr\omega_Q$ and $Q\in H^0(I_F(2))=\Ker(m)$ by hypotesis, one has
$$\sum_{k=1}^d\dfrac{m(G_Q^{-})}{\pi^*\alpha^2}(x_k)=-\sum_{k=1}^d\dfrac{m(\pi^*\alpha\spr\omega_Q)}{\pi^*\alpha^2}(x_k)=-\dfrac{Tr_{\pi}(\omega_Q)}{\alpha}(c).$$
But $Tr_{\pi}(\omega_Q)$ is an element of $H^0(\omega_E)$ so it is equal to $r\cdot \alpha$ for some $r$. Thus we have
$$\sum_{k=1}^d\dfrac{m(G_Q^{-})}{\pi^*\alpha^2}(x_k)= - r$$
which is zero if and only if $\omega_Q$ has trace $0$, i.e. if and only if $\omega_Q\in H^0(\omega_F)^{-}$. This happens if and only if $(\pi^*\alpha)^{\otimes 2}$ doesn't appear in the equation of $Q$, i.e. if and only if $q^{-}\in Q$.
\end{proof}

\noindent Using Lemma \ref{LEM:QUAKER} and Lemma \ref{LEM:FUNCTPOINT} the proof of Theorem \ref{THM:HALFGEO} is straightforward.

\begin{proof}[Proof of Theorem \ref{THM:HALFGEO}]
Assume that 
$$q^-\not\in \bigcap_{F\subset Q} Q.$$
Then, there exists a quadric which cointains $F$ but doesn't contain $q^{-}$. Denote by $G_Q$ its equation. By Lemma \ref{LEM:QUAKER} we know that $\beta=\gamma(G_Q)=G_Q^{-}\in \Ker(\dPEv)$ and by Lemma \ref{LEM:FUNCTPOINT} we have that 
$$\sum_{k=1}^d\dfrac{m(\beta)}{\pi^*\alpha^2}(x_k)\neq 0.$$ 
Hence, using Theorem \ref{THM:RESFORMULA} we have the thesis.
\end{proof}

\begin{rmk}
In \cite{Kan}, with different methods, it is proved that $H^0(I_F(2))=\Ker(\dPEv)$ if the ramification is simple. This fact is exactly what allows to prove the converse implication of Theorem \ref{THM:HALFGEO}.
\end{rmk}

\begin{rmk}
Notice that $H^0(I_F(2))=\Ker(\dPEv)$ if and only if all the ramification indices are equal to $2$. Indeed, denote by $R_{red}$ the reduced divisor whose support equals the support of the ramification divisor. Let $\bar{R}$ be $R-R_{red}$. From Riemann-Hurwitz we have
$$2g-2 = \deg(R)=\deg(R_{red})+\deg(\bar{R})=n+\deg(\bar{R}).$$
Hence, from Equation (\ref{EQ:DIMPED}) one has
$$\dim\Ker d_oP_E^{\vee}=h^0(I_F(2))+\deg(\bar{R}).$$
As $\bar{R}\geq 0$ and is trivial if and only if all the ramification indices are equal to $2$ the claim follows. In particular, the converse implication of (\ref{EQ:HALFGEO}) in Theorem \ref{THM:HALFGEO} holds for coverings whose ramification indices are all equal to $2$.
\end{rmk}

\noindent We conclude this section by proving the existence of an exact sequence which should help to measure, in a more intrinsic way, how much $H^0(I_F(2))$ and $\Ker(\dPEv)$ differ.

\begin{prop}
Under the hypotesis of Theorem \ref{THM:HALFGEO} there is an exact sequence
\begin{equation}
\xymatrix{
0 \ar[r]& H^0(I_F(2))\ar@{^{(}->}[r]^\gamma & \Ker(\dPEv)\ar[r] & \dfrac{\Ker(dh^{\vee})}{H^0(\omega_F)\spr\pi^*H^0(\omega_E)\ar[r]} & 0.
}
\end{equation}
\end{prop}

\begin{proof}
Starting from diagram (\ref{EQ:DIAG1}) it is easy to see that the composition of the inclusion of $H^0(\omega_F)\spr\pi^*H^0(\omega_E)$ with $m$ has image in $H^0(\omega_F^{\otimes 2})$ but also in the kernel of $dh^{\vee}$. Hence there is a map $$\epsilon:H^0(\omega_F)\spr\pi^*H^0(\omega_E)\rightarrow \Ker(dh^{\vee}),$$
which is easily proven to be injective as we have done with $\gamma$. 
We can also complete the diagram on the right by adding two (trivial) vertical arrows. 
The complete diagram looks like this
\begin{equation}
\label{EQ:BIGDIAG}
\xymatrix@C=20pt{
 &  & 0 & 
\\
0\ar[r] &
	\Ker\dPEv\ar@{^{(}->}[r]^-{j}  & 	
	\Sym^2(H^0(\omega_F)^-)\ar[u] \ar[r]^-{d_oP_E^{\vee}}  & T^{\vee}_{o}\cH_E \ar[r] & (\Ker \dPE)^{\vee}\ar[r] & 0
\\
0\ar[r] & H^0(I_F(2))\ar@{^{(}->}[r]^-{\iota}\ar@{^{(}->}[u]^-{\gamma} & 
\Sym^2(H^0(\omega_F))\ar@{->>}[u]^-{pr} \ar@{->>}[r]^-{m} & 
H^0(\omega_F^{\otimes 2})\ar[u]_{dh^{\vee}}\ar[r] & 0\ar[u]
\\
 & 0\ar[u]\ar[r] & H^0(\omega_F)\spr \pi^*H^0(\omega_E) \ar@{^{(}->}[r]^-{\epsilon}\ar@{^{(}->}[u] & \Ker dh^{\vee}\ar[u]\ar[r] & \Coker \epsilon\ar[r]\ar[u] & 0 
\\
 & & 0\ar[u] & 0\ar[u]
}
\end{equation}
By using the snake lemma on the central columns one obtain the wanted sequence. 
\end{proof}


\section{An interesting family of curves}
\label{SEC:4}

\noindent In this section we review the first example, due to Pirola, of a non-trivial family of coverings of elliptic curves with $2$ independent directions along which the Prym map $\Phi$ is constant. Hence the kernel of the differential of the Period map associated to $\Phi$ has dimension greater than $1$. The existence of the family is proved in \cite{Pir} but the proof is not constructive and uses a framework different form ours. After some notations and a brief idea of how to prove the existence of this family (for details, see \cite{Pir}), we will prove that $q^-$ belongs to the only quadric that contains $F$ and that for all the elements of $\Ker(\dPEv)$ the sum in Theorem \ref{THM:RESFORMULA} is $0$.
\vspace{2mm}

\noindent In order to prove the existence of such a family, let $G\simeq \Z_3$ and consider the space $\cH^G$ of Galois coverings $\pi:F\rightarrow E$ of degree $3$ with ramification given by $3$ points (so the number of branch points is exactly $3$ and the genus of $F$ is $4$) modulo the identifications given by a commutative diagram like
$$
\xymatrix{
F_1\ar[d]_{\pi_1}\ar@{-->}[r]^-{\simeq} & F_2\ar[d]^-{\pi_2} \\
E_1\ar@{-->}[r]_-{\simeq} & E_2
}
$$
With this type of identification of two coverings the dimension of $\cH^G$ is $3$. Note that, with this definition, a covering $\pi:F\rightarrow E$ and the covering obtained by composing $\pi$ with a translation of $E$ are equivalent: they represent the same point in $\cH^G$. 
\vspace{2mm}

\noindent Fix a generator $g$ of $G$ and $\rho$, a primitive root of $1$ of order $3$. If $V$ is a vector space on which $G$ acts, we will denote by $V_{\rho^k}$ the subspace where $g$ acts as the multiplication by $\rho^k$. As $\pi$ is the quotient by the group $G$, the $G$-action on $F$ induces several other $G$-actions. We will do now a small list of the one that we are going to use in this section.
\begin{enumerate}[a)]
\item The canonical action on $H^0(\omega_F)$ via pullback: by changing, if necessary, $g$ with $g^2$, we have
\begin{equation}
\label{EQ:DECH0}
H^0(\omega_F)=H^0(\omega_F)_{1}\oplus H^0(\omega_F)_{\rho}\oplus H^0(\omega_F)_{\rho^2}=\pi^*H^0(\omega_E)\oplus \C^{2}_{\rho}\oplus\C^{1}_{\rho^2}.
\end{equation}
\item \label{ENUM:2} A canonical $G$-action on $H^0(\omega_F)^-$ which is simply the restriction of the canonical representation on $H^0(\omega_F)$.
\item \label{ENUM:5} An action on $\Sym^2(H^0(\omega_F))$, whose decomposition in irreducible subrepresentations is given by
$$\Sym^2(H^0(\omega_F))=\C_{1}^{3}\oplus \C_{\rho}^{3}\oplus \C_{\rho^2}^{4}.$$
\item an action on $H^0(\omega_F^{\otimes 2})$ using the surjectivity of $m$ by imposing that $m$ becomes a morphism of $G$-vector spaces and hence on its dual $H^1(T_F)$.
\item \label{ENUM:4} An action on $H^0(I_F(2))$ as the kernel of $m$.
\item \label{ENUM:1} An action on the Prym $\Phi(\pi)$: this is induced at level of tangent spaces (as the tangent space $T_0\Phi(\pi)$ is $H^0(\omega_F)^-$) and it is compatible with the quotient by the periods' lattice.
\item \label{ENUM:3} An action of $G$ on $\PP H^0(\omega_F)^{\vee}=\PP$ as every automorphism of $F$, seen as a canonical curve in $\PP$ lifts to an automorphism of the whole space.
\end{enumerate}
All these actions, by construction, are compatible via the usual identification. For example, if we interpret $H^0(\omega_F)$ as the space of equations of hyperplanes of $\PP$ an invariant hyperplane in $\PP$ has an equation which is an eigenvector of $g$ in $H^0(\omega_F)$.
\vspace{2mm}

\noindent One has a Prym map $\tilde{\Phi}:\cH^G\rightarrow \A_{g-1}$ and a period map $\tilde{P}:\cH^G\rightarrow \mD$. We stress that, by construction, if we prove that $\dim(\Ker(d_{o}\tilde{P}))>k$ then, the period map $P$ associated to the Prym map of a local family of coverings with $\pi$ as central fiber will have kernel of dimension at least $k+1$.
\vspace{2mm}

\noindent The rough idea to prove that there exists a family of coverings in $\cH^G$ which gets contracted by $\tilde{\Phi}$ is to observe, as we have done in \ref{ENUM:1}), that the Prym map $\tilde{\Phi}$ factors through the inclusion of $\A_{g-1}^{G}$, the space of abelian varieties of dimension $g-1$ with an action of $G$, in $\A_{g-1}$. If we denote by $\mD^G$ a period domain for $\A_{g-1}^G$ we have an analogous period map $\tilde{P}^G:\cH^G\rightarrow \mD^G$. We want to get a bound on the dimension of the image $d\tilde{P}^G$. 
\vspace{2mm}

\noindent Clearly, the image of $d\tilde{P}^G$ has dimension at most the dimension of $$T\mD^{G}=\Sym^2(H^0(\omega_F)^-)^G$$
and the same bound holds, by construction, for the dimension of the image of $d\tilde{P}$.
As $\Sym^2(H^0(\omega_F)^-)^G$ is isomorphic, by \ref{ENUM:2}), to 
$$H^0(\omega_F)_{\rho}\spr H^0(\omega_F)_{\rho^2},$$
we have that its dimension is $2$. As $T_{[\pi]}\cH^G$ has dimension $3$ this implies that the kernel of $d\tilde{P}$ has dimension at least $1$ and the existence of the family is proved. 

\begin{prop}
\label{PROP:QUACONE}
Let $\pi:F\rightarrow E$ with $[\pi]\in \cH^G$, identify $F$ with its canonical model and let $Q$ be the only quadric containing $F$. Then $q^-\in Q$ and either $$H^0(I_F(2))\subset \Sym^2(H^0(\omega_F))_{\rho}\quad \mbox{ or }\quad H^0(I_F(2))\subset \Sym^2(H^0(\omega_F))_{\rho^2}.$$
\end{prop}

\begin{proof}
There exists only a quadric containing $F$ because $g(F)=4$. More precisely $F$ is the complete intersection of a quadric $Q$ and a cubic surface $C$. Let $G_Q\in H^0(I_F(2))$ be an equation for $Q$. Being $F$ invariant under the $G$-action introduced in \ref{ENUM:3}), we have that the orbit of $G_Q$ under the action given in \ref{ENUM:4}), is simply given by itself plus, possibly, some of its multiple by elements in $\C^*$. The key point now is to see that $H^0(I_F(2))_1=0$. In order to prove this observe that, by construction, we have an exact sequence of $G$-vector spaces given by
$$
\xymatrix{
0\ar[r] & H^0(I_F(2))\ar[r] & \Sym^2(H^0(\omega_F))\ar[r] & H^0(\omega_F^{\otimes 2})\ar[r] & 0
}
$$
Hence, by taking invariant parts and dimensions we have
$$\dim(H^0(I_F(2))^G)=\dim(\Sym^2(H^0(\omega_F))^G)-\dim(H^0(\omega_F^{\otimes 2})^G).$$
As claimed in \cite{Pir}, we can identify $T_{[\pi]}\cH^G$ with $H^1(T_F)^G=(H^0(\omega_F^{\otimes 2})^{\vee})^G$. Hence, we have $\dim(H^0(\omega_F^{\otimes 2})^G)=3$. Using \ref{ENUM:5}) we have that also $\dim(\Sym^2(H^0(\omega_F))^G)=3$ so, as claimed, $H^0(I_F(2))_1=0$.
\vspace{2mm}

\noindent As consequence we have either $G_Q\in \Sym^2(H^0(\omega_F))_{\rho}$ or $G_Q\in \Sym^2(H^0(\omega_F))_{\rho^2}$. Note that in both cases, as $\pi^*H^0(\omega_E)^2\subset \Sym^2(H^0(\omega_F))_{1}$, we have $G_Q(q^-)=0$ so $q^-\in Q$.
\end{proof}

\begin{Lem}
\label{LEM:G13s}
Let $\pi:F\rightarrow E$ with $[\pi]\in \caH^G$ and assume that $F$ is not hyperelliptic. Denote by $a_1,a_2$ and $a_3$ the $3$ ramification points of $\pi$. Let $L$ be a $g^1_3$. Then:
\begin{itemize}
\item $L$ is $G$-invariant, i.e. $g^*L=L$;
\item $h^0(\OO_F(3a_i))=1$;
\item If $L'$ is a $g_3^1$ then $L\simeq L'$, i.e. there is only one $g_3^1$ on $F$.
\end{itemize}
\end{Lem}

\begin{proof}
Recall that every curve of genus $4$ is trigonal and moreover, the number of $g^1_3$ is at most $2$. If there is only one $g_3^1$ clearly it is $G$-invariant. If there are $2$, as $G$ has order $3$ and acts on a set of two elements, it has to fix both of them.
\vspace{2mm}

\noindent Now let's prove that $h^0(\OO_F(3a_i))=1$. The Riemann-Roch formula for $\OO_{F}(3a_i)$ is
$$h^0(\OO_F(3a_i))-h^1(\OO_F(3a_i))=\deg(h^0(\OO_F(3a_i)))-4+1=0$$
so, by Serre duality, we have
$$h^0(\OO_F(3a_i))=h^0(\omega_F(-3a_i)).$$
From 
$$0\rightarrow \omega_F(-3a_i)\rightarrow \omega_F(-2a_i)\rightarrow \omega_F(-2a_i)|_{a_i}\rightarrow 0$$
one has $H^0(\omega_F(-3a_i))\leq H^0(\omega_F(-2a_i))$. In particular, as $F$ is not hyperelliptic we obtain that the dimension of $H^0(\omega_F(-3a_i))$ is either $1$ or $2$. Moreover, $h^0(\omega_F(-3a_i))=2$ if and only if $H^0(\omega_F(-3a_i))= H^0(\omega_F(-2a_i))$. But this cannot happen as the pullback $\eta$ of a non-zero holomorphic form on $E$ has a zero of multiplicity $2$ exactly in the ramification points so there is at least one element in $H^0(\omega_F(-2a_i))\setminus H^0(\omega_F(-3a_i))$. Hence $h^0(\OO_F(3a_i))=1$ as claimed.
\vspace{2mm}

\noindent Recall that on $F$ there are at most two $g_3^1$ and they are related by $$L\otimes L'=\omega_F=\OO_F(2a_1+2a_2+2a_3).$$ 
Hence we will conclude by proving that $L=\OO_F(a_1+a_2+a_3)$.
Let $A,B$ in $F$ such that $L\simeq \OO_F(a_1+A+B)$. As $a_1$ is invariant and the same holds for $L$, we have that $\OO_F(g(A)+g(B))=\OO_F(A+B)$. Therefore, as $F$ is not hyperelliptic, also the equality of divisors $g(A)+g(B)=A+B$ has to hold. Moreover, as $g$ has order $3$, it cannot exchange $A$ and $B$: we have proved that $A$ and $B$ are ramification points. If we assume that $L\neq \OO_F(a_1+a_2+a_3)$ there are several possibilities:
\begin{description}
\item [$A=a_1=B$] This is impossible as we would have $$2=h^0(L)=h^0(\OO_F(3a_1))=1.$$
\item [$A=a_1\neq B$] Assume that $A=a_1$ and $B=a_2$ so that $L\simeq \OO_{F}(2a_1+a_2)$. Let $C,D\in F$ such that $\OO_{F}(2a_1+a_2)\simeq \OO_{F}(a_3+C+D)$. As before, we have that $C$ and $D$ are ramification points and as $F$ is not hyperelliptic the only possible option is to have $C=D=a_3$. But then, again, we have a contradiction
$$2=h^0(\OO_{F}(L))=h^0(\OO_{F}(2a_1+a_2))=h^0(\OO_{F}(3a_3))=1.$$
\item [$A\neq a_1= B$] This case is analogous to the previous one.
\item [$A=B\neq a_1$] This case is analogous to the second one.
\end{description}
Hence, we have proved that $L=\OO_F(a_1+a_2+a_3)$ and thus that $L\simeq L'$ and there is only a $g_3^1$ on $F$. 
\end{proof}

\begin{prop}
Let $\pi:F\rightarrow E$ with $[\pi]\in \caH^G$ and assume that $F$ is not hyperelliptic. Denote by $Q$ be the only quadric containing the canonical curve $F$. Then $Q$ is a quadric cone with vertex $V$ and $V\not\in F$. Moreover, the hyperplane $H^-$ is tangent to the cone and the $3$ ramification points of $\pi$ lie on a line on the cone.
\end{prop}

\begin{proof}
\noindent Recall that if the quadric $Q$ containing $F$ is smooth, then $F$ can be seen as a curve of bidegree $(3,3)$ in $\PP^1\times \PP^1$ and the projections on each factor give two different $g^1_3$. If, instead, $Q$ is a cone (these are the only possible cases as $F$ is non degenerate) there exists only one $g^1_3$. Hence, by Lemma \ref{LEM:G13s}, we can conclude that $Q$ is a cone. If $V$ is the vertex, it is clear that $V\not\in F$ as, otherwise $F$ would be singular.
\vspace{2mm}

\noindent Now we will prove that the ramification points are on a line in the canonical model of $F$. By what we have seen in this section we have a decomposition of $H^0(\omega_F)$ into subrepresentations with $H^0(\omega_F)_1=\pi^*H^0(\omega_E)$. We can assume, as before, that $H^0(\omega_F)_{\rho}$ has dimension $2$. Denote respectively with $\{u_0\}$, $\{u_1,u_2\}$ and $\{u_3\}$ a basis for $\pi^*H^0(\omega_E)$, $H^0(\omega_F)_{\rho}$ and $H^0(\omega_F)_{\rho^2}$. By abuse of notation we will write $u_iu_j$ to mean $u_i\spr u_j$. With these coordinates, the hyperplane $H^-$ has equation $u_0=0$ and $q^-=(1:0:0:0)$. As
\begin{equation}
\Sym^2(H^0(\omega_F))=\langle u_0^2,u_1u_3,u_2u_3\rangle_{1}\oplus
\langle u_0u_1,u_0u_2,u_3^2\rangle_{\rho}\oplus
\langle u_0u_3,u_1^2,u_2^2,u_1u_2\rangle_{\rho^2}
\end{equation}
We know by Proposition \ref{PROP:QUACONE} that an equation $G_Q$ of $Q$ is either an element of $\Sym^2(H^0(\omega_F))_{\rho}$ or of $\Sym^2(H^0(\omega_F))_{\rho^2}$. In the first case the generic element of $\Sym^2(H^0(\omega_F))_{\rho}$ is a quadric cone and has equation
$$u_3^2+u_0(au_1+bu_2)=0.$$
Moreover, it is easy to see that $H^-$ is tangent to the cone along the line $L_1=\{u_3=u_0=0\}$. In the second case the generic element of $\Sym^2(H^0(\omega_F))_{\rho^2}$ is a smooth quadric but it is easy to see that the generic singular element is a cone with equation
$$u_3u_0+(au_1+bu_2)^2=0.$$
As before, $H^-$ is a plane tangent to $Q$ along the line $L_2=\{u_0=au_1+bu_2=0\}$. So, in both cases, as the ramification points of the canonical curve $F$ are given exactly as $H^-\cap F$, they are on a line as claimed. 
\end{proof}



\noindent Now we are going to calculate the sum in Theorem \ref{THM:RESFORMULA} and to see that it is zero for each element in $\Ker(\dPEv)$. 

\begin{prop}
Let $\pi:F\rightarrow E$ be a Galois covering of degree $3$ as before and consider a local family of coverings with central fiber $\pi$. Then, if 
$$\nu(\beta)=\sum_{k=1}^3\dfrac{m(\beta)}{\pi^*\alpha^2}(x_k),$$
one has $\nu(\beta)=0$ for all $\beta\in \Ker(\dPEv)$.
\end{prop}

\begin{proof}
$\Ker(\dPEv)$ is a subspace of $\Sym^2(H^0(\omega_F)^-)$. If $\beta\in \Ker(\dPEv)$ we can decompose it as
$$\beta=\beta_1+\beta_{\rho}+\beta_{\rho^2}$$
with $\beta_{\mu}\in \Sym^2(H^0(\omega_F)^-)_{\mu}$. First of all we will prove that $\nu(\beta_\rho)=\nu(\beta_{\rho^2})=0$.
\vspace{2mm}

\noindent As $c$ is not a branch point, we have that the fiber $\pi^{-1}(c)=\{x_1,x_2,x_3\}$ over $c$ is equal to the orbit of each of its points: $\pi^{-1}(c)=\{x_1,g(x_1),g^2(x_1)\}$.
Hence

$$\nu(\beta)=\sum_{k=1}^3\dfrac{m(\beta)}{\pi^*\alpha^2}(x_k)=\sum_{k=0}^2\dfrac{m(\beta)}{\pi^*\alpha^2}(g^k(x_1))=\sum_{k=0}^2\dfrac{m(\beta\circ g^k)}{\pi^*\alpha^2}(x_1).$$
If we assume that $\beta$ is in the eigenspace $\Sym^2(H^0(\omega_F))_{\mu}$ of $g^*$ then
$$\nu(\beta)=\dfrac{m\left(\sum_{k=1}^d(g^*)^k(\beta)\right)}{\pi^*\alpha^2}(x_1)
=\dfrac{m\left(\sum_{k=0}^2\mu^k\beta\right)}{\pi^*\alpha^2}(x_1)
=\left(\sum_{k=0}^2\mu^k\right)\dfrac{m(\beta)}{\pi^*\alpha^2}(x_1).$$
Hence, if $\mu\neq 1$, we have $\lambda(\beta)=0$ as claimed.
\vspace{2mm}

\noindent Hence we have that $\nu(\beta)=\nu(\beta_1)$ so it is enough to prove that 
$$\ker(\dPEv)\subseteq \Sym^2(H^0(\omega_F)^-)_{\rho}\oplus\Sym^2(H^0(\omega_F)^-)_{\rho^2}$$
i.e., that $\beta_1=0$.
\vspace{2mm}

\noindent Let $a$ be a ramification point and consider holomorphic coordinates $(U,z)$ centered in $a$ and $(V,w)$ centered in $\pi(a)=b$. Assume, moreover, that $\alpha|_V=dw$, the relation $w=z^3$ holds and the action of $g\in G$ near $a$ is given by $z\mapsto \rho z$ for $\rho\neq 1$ such that $\rho^3=1$. By changing $\rho$ with $\rho^2$ we can assume, moreover, that the decomposition of $H^0(\omega_F)$ in invariant subspaces with respect to the action of $G$ is the one given in Equation (\ref{EQ:DECH0}). Consider $\eta\in H^0(\omega_F)$. Near $a$ we can write
$$\eta|_U=\left(\sum_{j\geq 0}\eta_jz^j\right)dz\quad \mbox{ and }\quad g^*\eta|_U=\rho\left(\sum_{j\geq 0}\eta_j\rho^jz^j\right)dz.$$ 
In particular, $\eta\in H^0(\omega_F)^G$ if and only if, near $a$ we have
$$\eta|_U=\left(\sum_{j\geq 0}\eta_{2+3j}z^{2+3j}\right)dz$$
and an analogous decomposition holds near the other ramification points. Similarly, we have
$$\eta|_U=\left(\sum_{j\geq 0}\eta_{3j}z^{3j}\right)dz\quad\mbox{ and }\quad \eta|_U=\left(\sum_{j\geq 0}\eta_{1+3j}z^{1+3j}\right)dz$$
if $\eta\in H^0(\omega_F)_{\rho}$ and $\eta\in H^0(\omega_F)_{\rho^2}$ respectively. 
\vspace{2mm}

\noindent As
$$\Sym^2(H^0(\omega_F)^-)_{1}=H^0(\omega_F)_{\rho}\otimes H^0(\omega_F)_{\rho^2},\quad  \Sym^2(H^0(\omega_F)^-)_{\rho}=H^0(\omega_F)_{\rho^2}^{\otimes 2},$$
and
$$\Sym^2(H^0(\omega_F)^-)_{\rho^2}=\Sym^2(H^0(\omega_F)_{\rho}),$$
if $\varphi\in \Sym^2(H^0(\omega_F)^-)_\mu$ 
we can write it in coordinate near $a$ as
$$\varphi|_U=z\left(\varphi_0+\varphi_1z^3+o(z^5)\right)dz^2$$
for $\mu=1$ and as
$$\varphi|_U=\left(\varphi_0+\varphi_1z^3+o(z^5)\right)dz^2\quad \mbox{ and }\quad \varphi|_U=z^2\left(\varphi_0+\varphi_1z^3+o(z^5)\right)dz^2$$
if $\mu=\rho$ and $\mu=\rho^2,$ respectively. In the latter cases, we have that the residue of $\varphi/\pi^*\alpha$ in $a$ is $0$ as $\varphi/\pi^*\alpha$ is either holomorphic or has a pole of order $2$ with coefficient of degree $-1$ equal to $0$. Hence
$$\Sym^2(H^0(\omega_F)^-)_{\rho}\oplus \Sym^2(H^0(\omega_F)^-)_{\rho^2}\subseteq \Ker(\dPEv).$$
By Theorem \ref{THM:KERDPHIE} and using Diagram (\ref{EQ:BIGDIAG}) we obtain $\dim(\Ker(\dPEv))=4$. This is equal to the dimension of $\Sym^2(H^0(\omega_F)^-)_{\rho}\oplus \Sym^2(H^0(\omega_F)^-)_{\rho^2}$ so
$$\Sym^2(H^0(\omega_F)^-)_{\rho}\oplus \Sym^2(H^0(\omega_F)^-)_{\rho^2}=\Ker(\dPEv).$$
In particular, $\beta_1=0$ and $\nu(\beta)=0$ as claimed.

\end{proof}


\end{document}